\documentclass[11pt]{article}

\usepackage{amsmath,amssymb,amsthm,setspace,graphicx}
\usepackage[text={6.5in,8.4in}, top=1.3in]{geometry}

\newtheorem{theorem}{Theorem}[section]
\newtheorem{proposition}[theorem]{Proposition}
\newtheorem{lemma}[theorem]{Lemma}

\newtheorem{problem}[theorem]{Problem}
\newtheorem{definition}[theorem]{Definition}

\newcommand{\D}{\mathcal{D}}

\begin{document}
	
\title{
	The covering threshold of a directed acyclic graph by directed acyclic subgraphs}
	
\author{
	Raphael Yuster
	\thanks{Department of Mathematics, University of Haifa, Haifa 3498838, Israel. Email: raphael.yuster@gmail.com\;.}
}
	
\date{}
	
\maketitle
	
\setcounter{page}{1}
	
\begin{abstract}
Let $H$ be a directed acyclic graph (dag) that is not a rooted star.
It is known that there are constants $c=c(H)$ and $C=C(H)$ such that the following holds
for $D_n$, the complete directed graph on $n$ vertices.
There is a set of at most $C\log n$ directed acyclic subgraphs of $D_n$ that covers every $H$-copy of $D_n$,
while every set of at most $c\log n$ directed acyclic subgraphs of $D_n$ does not cover all $H$-copies.
Here this dichotomy is considerably strengthened.

Let ${\vec G}(n,p)$ denote the Erd\H{o}s-R\'enyi and Gilbert probability space of all directed graphs with $n$ vertices and with edge probability $p$.
The {\em fractional arboricity} of $H$ is $a(H) = max \{\frac{|E(H')|}{|V(H')|-1}\}$, where the
maximum is over all non-singleton subgraphs of $H$.
If $a(H) = \frac{|E(H)|}{|V(H)|-1}$ then $H$ is {\em totally balanced}. All complete graphs, complete multipartite graphs, cycles, trees, and, in fact, almost all graphs, are totally balanced. We prove:

\begin{itemize}
	\item
	Let $H$ be a dag with $h$ vertices and $m$ edges which is not a rooted star.
	For every $a^* > a(H)$ there exists $c^* = c^*(a^*,H) > 0$ such that asymptotically almost surely $G \sim {\vec G}(n,n^{-1/a^*})$ has the property that every set $X$ of at most $c^*\log n$ directed acyclic subgraphs of $G$ does not cover all $H$-copies of $G$.\\
	Moreover, there exists $s(H) = m/2 + O(m^{4/5}h^{1/5})$ such that the following stronger assertion holds for any such $X$: 	
	There is an $H$-copy in $G$ that has no more than $s(H)$ of its edges covered by each element of $X$.
	\item 
	If $H$ is totally balanced then for every $0 < a^* < a(H)$, asymptotically almost surely $G \sim {\vec G}(n,n^{-1/a^*})$ has a single directed acyclic subgraph that covers all its $H$-copies.
\end{itemize}
As for the first result, note that if $h=o(m)$ then $s(H)=(1+o_m(1))m/2$ is essentially half of the edges of $H$. In fact, for infinitely many $H$ it holds that $s(H)=m/2$, optimally.
As for the second result, the requirement that $H$ is totally balanced cannot, generally, be relaxed.

\vspace*{3mm}

\end{abstract}

\section{Introduction} 

Our main objects of study are simple finite directed graphs. We denote by $D_n$ 
the complete directed $n$-vertex graph consisting of all possible $n(n-1)$ edges.
An important subclass of directed graphs are {\em directed acyclic graphs} (hereafter, a directed acyclic graph is called a {\em dag}) which are directed graphs with no directed cycles.
The largest dag on $n$ vertices is the transitive tournament, denoted here by $T_n$.

It is easily observed that the edge-set of every directed graph $G$ is the disjoint union of two dags.
Indeed, consider some permutation $\pi$ of $V(G)=[n]$.
Let $G_L(\pi)$ be the spanning subgraph of $G$ where $(i,j)\in E(G_L(\pi))$ if and only if $(i,j) \in E(G)$ and $\pi(i) < \pi(j)$.
Let $G_R(\pi)$ be the spanning subgraph of $G$ where $(i,j)\in E(G_R(\pi))$ if and only if $(i,j) \in E(G)$ and $\pi(i) > \pi(j)$.
Since $E(G_R(\pi)) \cup E(G_L(\pi)) = E(G)$, we can cover the edges of $G$ using just two
dag-subgraphs of $G$.

However, the aforementioned edge-covering observation becomes more involved if instead of just covering edges, we aim to cover all given $H$-subgraphs\footnote{To avoid trivial cases, we assume hereafter that $H$ has at least two edges and no isolated vertices.} of $G$ with as few as possible dag-subgraphs.
Of course, for this to be meaningful we assume that $H$ is a dag.
More formally, we say that a subgraph $H^*$ of $G$ that is isomorphic to $H$ (i.e. an $H$-copy of $G$)
is {\em covered} by $\pi$, if $H^*$ is a subgraph of $G_L(\pi)$.
What is the minimum number of permutations required to cover all $H$-copies of $G$?
Let $\tau(H,G)$ be the smallest integer $t$ such that the following holds:
There are permutations $\pi_1,\ldots,\pi_t$ of $V(G)$ such that each $H$-copy of $G$ is
covered by at least one of the $\pi_i$.
Trivially, $\tau(H,G)$ exists as we can just consider all possible permutations and use the fact that each $H$-copy, being a dag, has a topological ordering.

While determining $\tau(H,G)$ is generally a difficult problem, reasonable bounds are known
for $\tau(H,D_n)$. In fact, $\tau(T_h,D_n)$ is equivalent to a well-studied problem in the setting of permutations. An {\em $(n,h)$-sequence covering array} (SCA) is a set $X$ of permutations of $[n]$ such that
each sequence of $h$ distinct elements of $[n]$ is a subsequence of at least one of the permutations.
So, clearly, $\tau(T_h,D_n)$ is just the minimum size of an $(n,h)$-SCA.
The first to provide nontrivial bounds for $\tau(T_h,D_n)$ was Spencer \cite{spencer-1972} and various improvements on the upper and lower bounds were sequentially obtained by Ishigami \cite{ishigami-1995,ishigami-1996}, F\"uredi \cite{furedi-1996}, Radhakrishnan \cite{radhakrishnan-2003}, Tarui
\cite{tarui-2008} and the author \cite{yuster-2020}.
See also the paper \cite{CCHZ-2013} for further results and references to many applications.
The asymptotic state of the art regarding $\tau(T_3,D_n)$ is the upper bound of Tarui \cite{tarui-2008}
and the lower bound of F\"uredi \cite{furedi-1996}:
\begin{equation}\label{e:1}
\frac{2}{\log e} \log n \le \tau(T_3,D_n)  \le (1+o_n(1))2 \log n\;. \footnote{Unless stated otherwise, all logarithms are in base $2$.}
\end{equation}
For general fixed $h$, the best asymptotic upper and lower bounds are that of the author
\cite{yuster-2020} and Radhakrishnan \cite{radhakrishnan-2003}, respectively:
\begin{equation}\label{e:2}
(1-o_n(1))\frac{(h-1)!}{\log e} \log n \le \tau(T_h,D_n) \le \ln 2 \cdot h!(h-1) \log n + C_h\;.
\end{equation}
It is immediate to see that if $H$ has $h$ vertices, then $\tau(H,D_n) \le \tau(T_h,D_n)$, hence \eqref{e:1} and (\ref{e:2}) serve as upper bounds for $\tau(H,D_n)$ when $H$ has three vertices or, respectively, 
$h$ vertices. In particular, we have that $\tau(H,D_n) = O(\log n)$. However, for some $H$, this upper bound is far from optimal. Suppose that $H$ is a rooted star, meaning that $H$ is a star and the center of the star is either a source or a sink. It is proved in \cite{FHRT-1992,spencer-1972} that for such $H$ 
\begin{equation}\label{e:3}
\tau(H,D_n) = \Theta(\log\log n)\;.
\end{equation}
The  permutations in the upper bound construction of \eqref{e:3} are such that for any $v_1 \in [n]$
and any $h-1$ distinct elements $v_2,\ldots,v_h \in [n] \setminus \{v_1\}$ there is a permutation in which
$v_1$ appears before all of $v_2,\ldots,v_h$. As it turns out, rooted stars are the {\em only} dags for which $\tau(H,D_n)$ is sub-logarithmic.
\begin{theorem}\label{t:0}
	Let $H$ be a dag that is not a rooted star. Then $\tau(H,D_n)=\Theta(\log n)$.
\end{theorem}
The upper bound in Theorem \ref{t:0}  follows from the aforementioned fact that $\tau(H,D_n) = O(\log n)$
while the lower bound follows as a special case of Theorem \ref{t:1} stated below.

Our main goal is to determine the extent of which Theorem \ref{t:0} generalizes to directed $n$-vertex graphs other than $D_n$. A partial answer is given in \cite{yuster-2020} where it is proved that
$\tau(H,G)=\Theta(\log n)$ for {\em some} dags $H$, and for some {\em tournaments}\, $G$.
Here we prove that Theorem \ref{t:0} holds for {\em all} dags $H$ that are not rooted stars, while $G$ is allowed to be almost every directed graph that is not too sparse.

To state our main results we require some definitions and notations.
Let ${\vec G}(n,p)$ denote the Erd\H{o}s-R\'enyi and Gilbert probability space of all directed graphs with $n$ vertices and edge probability $p$. In other words, a sampled graph $G \sim {\vec G}(n,p)$ has vertex set $[n]$
and each ordered pair of vertices $(i,j)$ is chosen to be an edge of $G$ with probability $p$, where all $n(n-1)$ choices are independent. We observe that $\{D_n\}$ is just the (trivial) sample space of ${\vec G}(n,1)$.
As usual in the setting of random graphs, we say that a property of ${\vec G}(n,p)$ holds {\em asymptotically almost surely} (henceforth {\em almost surely}) for $p=p(n)$,
if the probability that $G \sim {\vec G}(n,p)$ has that property approaches $1$ as $n$ goes to infinity.
So, the natural way to extend Theorem \ref{t:0} is to ask, for a given dag $H$, how small can $p$ be such that it still holds almost surely for $G \sim {\vec G}(n,p)$ that $\tau(H,G)=\Theta(\log n)$. Furthermore, what happens if we
decrease that $p$ even further? Does $\tau(H,G)$ just gradually decrease below $\log n$, or does it quickly become constant (or even $1$), almost surely?  To address this question, we need the following definition.
\begin{definition}[fractional arboricity; totally balanced graph; maximal density]\label{def:frac-arbor}
	The {\em fractional arboricity} of a simple (directed or undirected) graph $H$ is $a(H) = max \{\frac{|E(H')|}{|V(H')|-1}\}$, where the
	maximum is taken over all non-singleton subgraphs of $H$.
	If $a(H) = \frac{|E(H)|}{|V(H)|-1}$ then $H$ is {\em totally balanced}.
	The {\em maximal density} of $H$ is $\rho(H) =  max \{\frac{|E(H')|}{|V(H')|}\}$, where the
	maximum is taken over all subgraphs $H'$ of $H$.
\end{definition}
Recall that by the seminal paper of Erd\H{o}s and R\'enyi \cite{ER-1960}, $n^{-1/\rho(H)}$ is the threshold function for the existence of an $H$-copy in a random graph.
By a well-known theorem of Nash-Williams \cite{nash-1964},
$\lceil a(H) \rceil$ is the arboricity of $H$, i.e., the minimum number of forests whose union covers all edges of $H$. Fractional arboricity appears in the definition of threshold functions of several important graph properties (see, e.g., \cite{AY-1993}). Observe also that the fractional arboricity of forests is $1$ while for $K_h$ (hence also $T_h$) it is
$h/2$. It is also easy to verify that all complete graphs, complete multipartite graphs, cycles, trees and many other families of graphs are totally balanced. In fact, it is a relatively simple exercise to prove the following:
\begin{proposition}\label{prop:1}
	Let $H \sim G(h,\frac{1}{2})$ \footnote{In Proposition \ref{prop:1} we consider simple undirected graphs.
			Observe that the underlying undirected graph of every dag is simple and hence a dag is totally balanced if its corresponding undirected underlying graph is totally balanced.}. The probability that $H$ is totally balanced is $1-o_h(1)$.
\end{proposition}

As it turns out, if $p$ is just barely larger than $n^{-1/a(H)}$, then almost surely
$G \sim {\vec G}(n,p)$ satisfies $\tau(H,G)=\Theta(\log n)$. This follows from our first main result.
\begin{theorem}\label{t:1}
	Let $H$ be a dag which is not a rooted star and let $a^* > a(H)$.
	There is a constant $c^* = c^*(a^*,H) > 0$ such that almost surely $G \sim {\vec G}(n,n^{-1/a^*})$ has $\tau(H,G) \ge c^*\log n$. In particular, almost surely $\tau(H,G) = \Theta(\log n)$.
\end{theorem}
Note that the ``in particular'' part of Theorem \ref{t:1} follows from the aforementioned fact that
$\tau(H,G) \le \tau(H,D_n)=O(\log n)$. Also observe that Theorem \ref{t:1} immediately implies Theorem \ref{t:0}.

So, what happens if $p$ is just barely smaller than $n^{-1/a(H)}$? \footnote{If $p$ is smaller than $n^{-1/\rho(H)}$ then almost surely $G \sim {\vec G}(n,p)$ will have no copy of $H$, so trivially $\tau(H,G)=0$.} For totally balanced graphs, the situation changes drastically; almost surely, $G \sim {\vec G}(n,p)$
becomes at most $1$.
\begin{theorem}\label{t:2}
	Let $H$ be a totally balanced dag which is not a rooted star and let $0 < a^* < a(H)$.
	Almost surely $G \sim {\vec G}(n,n^{-1/a^*})$ has
	$\tau(H,G) \le 1$.
	In particular, if $\rho(H) < a^* < a(H)$ then almost surely $\tau(H,G) = 1$.
\end{theorem}
Note that the ``in particular'' part of Theorem \ref{t:2} follows from the aforementioned result \cite{ER-1960}.
It is important to note that we cannot relax the requirement in Theorem \ref{t:2} that $H$ is totally balanced.
Indeed, as we later demonstrate, there are dags $H$ that are not totally-balanced for which 
almost surely $\tau(H,G) = \Omega(\log n)$ in the probability regime of Theorem \ref{t:2}.

We can strengthen Theorem \ref{t:1} even more as follows. Theorem \ref{t:1} says that for a typical $G \sim {\vec G}(n,n^{-1/a^*})$, every set of at most $c^*\log n$ acyclic subgraphs of $G$ (equivalently, permutations of $n$) fails to cover some $H$-copy of $G$. But perhaps this is just barely so? Perhaps there are $o(\log n)$ permutations that
have the property that for every $H$-copy, at least one of the permutations covers most of the edges of that copy?
This question is motivated by the fact that already a set of two permutations has the property that every $H$-copy has at least half of its edges covered by one of the permutations as observed by taking any permutation $\pi$ and its reverse $\pi^{rev}$ since $G_R(\pi)=G_L(\pi^{rev})$. For some $H$, a very small amount of permutations has the ``large coverage'' property. For example, let $H$ be obtained from a rooted star
with $h$ edges by adding $k \ll h$ edges. Then, $H$ has $h+k$ edges
but by \eqref{e:3} already $O(\log \log n)$ permutations suffice so that for each $H$-copy, at least $h$ (that is, most) of its edges are covered by some permutation. Is this, in a sense, a ``worst'' scenario, i.e. is it true that if $H$ is far from being a rooted star, we cannot ask for much more than $50\%$ coverage?  Indeed, this turns out to be the case.

\begin{theorem}\label{t:3}
	Let $H$ be a dag with $h$ vertices and $m$ edges which is not a rooted star. Then there exists
	$s(H) = m/2 + O(m^{4/5}h^{1/5}) < m$ such that the following holds for every $a^* > a(H)$:
	There is a constant $c^* = c^*(a^*,H) > 0$ such that almost surely $G \sim {\vec G}(n,n^{-1/a^*})$ has the property that for every set $X$ of at most $c^*\log n$ permutations, there is an $H$-copy of $G$ such that each element of $X$ contains at most $s(H)$ edges of that copy.
\end{theorem}
Notice that Theorem \ref{t:3} implies Theorem \ref{t:1} since $s(H) < m$, so it suffices to
prove Theorem \ref{t:3}. Observe that Theorem \ref{t:3} is much stronger than Theorem \ref{t:1} for dags $H$
having $h = o(m)$, since in this case $s(H)=(1+o_m(1))m/2$ is essentially half of the edges of $H$.
Recalling that already two permutations have the property that one of them covers at least half of the edges of an $H$-copy, we have that $s(H) \ge \lceil m/2 \rceil$ for every $H$.

The parameter $s(H)$ which can be defined for every digraph $H$ and which we call the {\em skewness} of $H$
(we defer its precise definition to Section \ref{s:skew}) may be of independent interest.
At this point, we should say that for infinitely many dags $H$ we can, in fact, prove that $s(H)=m/2$, optimally,
as this means that for any given $X$ as in Theorem \ref{t:3}, there is an $H$-copy such that no element of $X$ covers more than $50\%$ of the edges of $H$. For example, we show that $s(H)=m/2$ for every bipartite dag $H$ with a bipartition in which half of the edges go from one part to the other (particularly, a directed path of even length has this property).

The rest of this paper is organized as follows. In Section \ref{s:t-2} we prove our first main result,
Theorem \ref{t:2}. In Section \ref{s:skew} we define the aforementioned skewness $s(H)$, prove that 
$s(H) = m/2 + O(m^{4/5}h^{1/5}) < m$, and determine it for a few basic classes of dags.
In Section \ref{s:t-3} we prove our second main result, Theorem \ref{t:3}.
Section \ref{s:final} contains some concluding remarks and open problems. In particular,
we show there that there are dags $H$ that are not totally-balanced for which 
almost surely $\tau(H,G) = \Omega(\log n)$ even when $p \ll n^{-1/a(H)}$.

\section{Proof of Theorem \ref{t:2}}\label{s:t-2}

Fix a totally balanced dag $H$ that is not a rooted star and fix $0 < a^* < a(H)$. Let $h=|V(H)|$ and let $m=|E(H)|$. Observe that since $H$ is totally balanced, we have, in particular, $m=a(H)(h-1)$.
We shall prove that almost surely $G \sim {\vec G}(n,p)$ has
$\tau(H,G) \le 1$ where $p=n^{-1/a^*}$.
Let $G_H$ be the spanning subgraph of $G$ consisting of all edges that belong to at least one $H$-copy of $G$.
We must therefore show that with high probability, $G_H$ is a dag.

Recall that the girth of a directed graph is the length of a shortest directed cycle (so the girth of a dag is infinity).
For a set of (not necessarily edge-disjoint) graphs $S$, let $D(S)$ denote the graph obtained by the union of the elements of $S$.

For $k \ge 2$, a {\em $k$-cycle configuration} is a set $S$ of (not necessarily edge-disjoint)
copies of $H$ such that the girth of $D(S)$ is $k$ and which is minimal in the sense that the union of any proper subset of $S$ has girth larger than $k$. Observe that $2 \le |S| \le k$ and hence $D(S)$ has at most $km$ edges and at most $kh$ vertices.

For $k \ge 1$, a {\em $k$-path configuration} is a set $S$ of (not necessarily edge-disjoint)
copies of $H$ such that $D(S)$ contains an {\em induced} path of length $k$ and which is minimal
in the sense that the union of the elements of any proper subset of $S$ has no induced path of length $k$.
Observe that $1 \le |S| \le k$ and hence $D(S)$ has at most $km$ edges and at most $kh$ vertices.

Let $k_0$ be a positive integer to be set later, and let $\D$ be the set of all pairwise non-isomorphic
directed graphs $D$ such that $D=D(S)$ for some $k$-cycle configuration $S$ with $k \le k_0$.
Observe that $|\D| \le 2^{k_0^2h^2}$ since each element in it has at most $k_0h$ vertices
and there are at most $2^{k_0^2h^2}$ possible non-isomorphic graphs on at most $k_0h$ vertices.
Let $\D^*$ be the set of all pairwise non-isomorphic directed graphs $D$ such that $D=D(S)$ for some $k_0$-path configuration $S$. Observe that $|\D^*| \le 2^{k_0^2h^2} $ as each element in it has at most $k_0h$ vertices.

\begin{lemma}\label{l:config}
	If $G_H$ has no subgraph that is isomorphic to an element of $\D \cup \D^*$ then $G_H$ is a dag.
\end{lemma}
\begin{proof}
	Assume that $G_H$ is not a dag and let $C$ be a shortest directed cycle in $G_H$, denoting $|C|=k \ge 2$.
	Observe that $C$ is induced, as it is a shortest directed cycle.
	
	Consider first the case where $k \le k_0$. Let $S$ be a minimal set of $H$-copies in $G_H$
	such that $D(S)$ has girth $k$ and the union of any proper subset of $S$ has girth larger than $k$.
	Then $D(S) \in \D$ and $D(S)$ is a subgraph of $G_H$.
	
	Consider next the case $k > k_0$. Let $(v_0,\ldots,v_{k-1})$ be a consecutive ordering of the vertices of $C$.
	So, $P=v_0,\ldots,v_{k_0}$ is an induced path in $G_H$ of length $k_0$. We shall construct a $k_0$-path configuration.
	We sequentially construct, in at most $k_0$ steps, a set $S$ of $H$-copies in $G_H$ where, eventually, $S$ will be a $k_0$-path configuration. We initially define $S = \emptyset$.
	Before each step, $S$ will have the property that none of its subsets (including $S$ itself) is a $k_0$-path configuration (so this holds before the first round as $S=\emptyset$). A step is performed as follows.
	Let $j < k_0$ be the smallest index such that the edge $(v_j,v_{j+1})$ is not in any element of $S$
	(so, in the first step we have $j=0$). Let $H_j$ be an $H$-copy in $G_H$ that contains $(v_j,v_{j+1})$
	and extend $S$ by adding $H_j$ to it, thereby completing the current step.
	Now, if $D(S)$ has an induced path of length $k_0$ then $S$ (or some subset of $S$ that contains $H_j$) is a $k_0$-path configuration. Otherwise, we proceed to the next step.
	Notice that this procedure indeed halts after at most $k_0$ steps, as $k_0$ is the length of $P$.
	Also note that $D(S)$ is a subgraph of $G_H$ as it is the union of $H$-subgraphs of $G$.
\end{proof}

\begin{lemma}\label{l:none-1}
	Let $D \in \D$. Then the probability that $G$ contains a subgraph isomorphic to $D$ is $o_n(1)$.
\end{lemma}
\begin{proof}
	Fix $D \in \D$ with $D=D(S)$ where $S$ is a $k$-cycle configuration where $2 \le k \le k_0$.
	So, $S$ consists of copies of $H$ in $G$ with $2 \le |S|= t \le k$.
	Furthermore, $S$ is minimal in the sense that removing any element from it causes the union of the remaining elements to be a digraph with girth larger than $k$ (possibly infinite girth).
	In particular, the girth of $D$ is obtained by a directed $k$-cycle $C$ such that each element of $S$ contains
	at least one edge which belongs to $C$ and which does not belong to the other elements of $S$.
	Let $C=v_1,\ldots,v_k$ where $(v_i,v_{i+1}) \in E(D)$ (indices taken modulo $k$).
	Since $H$ is acyclic and since $C$ is a directed cycle, there is at least one vertex of $C$ that belongs to
	at least two distinct elements of $S$. Without loss of generality, $v_1$ is such a vertex.
	
	Let us totally order the elements of $S$ by $H_1,\ldots,H_t$ where the ordering is described below.
	It consists of four stages, where in the first stage we determine $H_1$. In the second stage
	we determine in sequence $H_2 \ldots H_r$. The third stage determines $H_{r+1}$.
	The fourth stage determines the remaining $H_{r+2},\ldots, H_t$. It may be that $r=1$ (in which case the second stage is empty) and it may be that $r=t-1$ (in which case the fourth stage is empty).
	
	First stage: Let $H_1 \in S$ with $v_1 \in V(H_1)$. Also, let $Q_1=V(H_1) \setminus \{v_1\}$ and let $h_1=|Q_1|=h-1$.
	
	Second stage: As long as there is some element $X \in S$ such that $v_1 \notin V(X)$ and
	$V(X) \cap (\cup_{j=1}^{i-1} V(H_j)) \neq \emptyset$ (so $X$ has at least one vertex other than $v_1$
	in common with at least one of the previously ordered elements), then let $H_i=X$.
	Also, let $Q_i = V(H_i) \setminus \cup_{j=1}^{i-1} V(H_j)$ be the set of vertices of $H_i$ not appearing in
	previous elements. Observe that $h_i = |Q_i| \le h-1$.
	
	Third stage: We have now reached the first time where we cannot apply the second stage.
	Observe that we have not yet ordered all elements since $v_1$ is a vertex of at least two elements of
	$S$, and only $H_1$ is a previously ordered element which contains $v_1$. 
	Suppose $H_2,\ldots,H_r$ were determined at the second stage where $1 \le r \le t-1$, and we now need to determine $H_{r+1}$. Let $H_{r+1}$  be one of the yet unordered elements.
	We claim that $H_{r+1}$ must contain $v_1$ and must contain at least one additional vertex appearing in previously ordered elements. Indeed, consider all edges of $C$ belonging $H_{r+1}$ and not to any previously
	order elements (recall that each element of $S$ contains at least one edge which belongs to $C$ and which
	does not belong to the other elements of $S$). As these edges form a union of disjoint nonempty paths,
	any such path has an endpoint which is not $v_1$ (every nonempty path has two endpoints). So this endpoint,
	call it $v_z$, belongs to a previously order element. Now if $v_1$ was not a vertex of $H_{r+1}$,
	we wouldn't have ended the second stage, as we could have picked $H_{r+1}$ in the second stage.
	So, both $v_1,v_z$ are distinct vertices of $H_{r+1}$ appearing in previous elements.
	Let $Q_{r+1} = \{v_1\} \cup (V(H_{r+1}) \setminus \cup_{j=1}^{r} V(H_j))$ be the set of vertices of $H_{r+1}$ not appearing in previous elements, in addition to $v_1$ which is included in $Q_{r+1}$.
	Observe that $h_{r+1} = |Q_{r+1}| \le h-1$.
	
	Fourth stage: Suppose we have already determined $H_1,\ldots,H_{i-1}$ with $i \ge r+2$.
	Now, if $i=t+1$ we are done. Otherwise, we determine
	$H_i$ as follows. Let $X \in S$ be any yet unordered element such that $V(X) \cap (\cup_{j=1}^{i-1} V(H_j)) \neq \emptyset$. Observe that there is at least one such element since $C$ is not yet covered.
	Let $H_i = X$ and let $Q_i = V(H_i) \setminus \cup_{j=1}^{i-1} V(H_j)$ be the set of vertices of $H_i$ not appearing in previous elements. Observe that $h_i = |Q_i| \le h-1$.
	
	Having completed the ordering $H_1,\ldots,H_t$ we can now evaluate the number of vertices and edges of $D$
	in terms of $h_1,\ldots,h_t$. First observe that $Q_1,\ldots,Q_t$ are pairwise disjoint sets which together
	partition the vertex set of $D$. So,
	$$
	|V(D)| = \sum_{i=1}^t h_i\;.
	$$
	To lower-bound $|E(D)|$ we proceed as follows. For $1 \le i \le t$, let $D_i$ be the union of $H_1,\ldots,H_i$
	(so $D=D_t$). We claim that $|E(D_i) \setminus E(D_{i-1})| \ge a(H)h_i$ (define $E(D_0)=\emptyset$).
	Observe first that this holds for $i=1$ since $|E(D_1)|=|E(H_1)|=m=(h-1)a(H)=h_1a(H)$.
	Suppose that $2 \le i \le t$. So,
	$E(D_i) \cap E(D_{i-1})$ is a set of edges of a subgraph of $H_i$ on $h-h_i$ vertices.
	Now, if $h_i=h-1$ then this subgraph is empty, so
	$|E(D_i) \setminus E(D_{i-1})|=|E(H_i)|=m=(h-1)a(H)=h_ia(H)$. Otherwise, by the definition of $a(H)$,
	this subgraph has at most $(h-h_i-1)a(H)$ edges, so
	$|E(D_i) \setminus E(D_{i-1})| \ge m - (h-h_i-1)a(H) = h_ia(H)$.
	We have now shown that
	$$
	|E(D)| = \sum_{i=1}^t |E(D_i) \setminus E(D_{i-1})| \ge  \sum_{i=1}^t a(H)h_i = a(H)|V(D)|\;.
	$$
	The probability that $G$ has a particular subgraph with $|V(D)|$ vertices and at least $a(H)|V(D)|$
	edges is at most 
	$$
	n^{|V(D)|}p^{a(H)|V(D)|} = n^{-|V(D)|(\frac{a(H)}{a^*}-1)} \le n^{-(\frac{a(H)}{a^*}-1)} = o_n(1)\;.
	$$
	\end{proof}
	
	\begin{lemma}\label{l:none-2}
		Let $D \in \D^*$. Then the probability that $G$ contains a subgraph isomorphic to $D$ is $o_n(1)$.
	\end{lemma}
\begin{proof}
	Fix $D \in \D^*$ with $D=D(S)$ where $S$ is a $k_0$-path configuration.
	So, $S$ consists of copies of $H$ in $G$ with $1 \le |S|= t \le k_0$.
	Furthermore, $S$ is minimal in the sense that removing any element from it causes the union of the remaining elements to be a digraph with no induced path of length $k_0$.
	Let $P=v_0,\ldots,v_{k_0}$ be an induced path in $D$ of length $k_0$.
	
	Let us totally order the elements of $S$ by $H_1,\ldots,H_t$ where the ordering is described below.
	Let $H_1 \in S$ be any element that contains the edge $(v_0,v_1)$. Let $Q_1=V(H_1)$ and let $h_1=|Q_1|=h$. Suppose we have already determined $H_1,\ldots,H_{i-1}$ where $2 \le i \le t$.
	Let $j < k_0$ be the smallest index such that the edge $(v_j,v_{j+1})$ is not in any element of $H_1,\ldots,H_{i-1}$ (such an edge exists by the minimality of $S$).
	Let $H_i \in S$ be any element that contains $(v_j,v_{j+1})$.
	Let $Q_i = V(H_i) \setminus \cup_{j=1}^{i-1} V(H_j)$ be the set of vertices of $H_i$ not appearing in previous elements. Observe that $h_i = |Q_i| \le h-1$ since the vertex $v_j$ is in $V(H_i)$, but since
	the edge $(v_{j-1},v_j)$ is an edge of some previously ordered element, $v_j \notin Q_i$.
	
	Having completed the ordering $H_1,\ldots,H_t$ we can now evaluate the number of vertices and edges of $D$
	in terms of $h_1,\ldots,h_t$. First observe that $Q_1,\ldots,Q_t$ are pairwise disjoint sets which together
	partition the vertex set of $D$. So,
	$$
	|V(D)| = \sum_{i=1}^t h_i\;.
	$$
	To lower-bound $|E(D)|$ we proceed as follows. For $1 \le i \le t$, let $D_i$ be the union of $H_1,\ldots,H_i$
	(so $D=D_t$). We claim that for all $2 \le i \le t$ it holds that $|E(D_i) \setminus E(D_{i-1})| \ge a(H)h_i$.
	Indeed, $E(D_i) \cap E(D_{i-1})$ is a set of edges of a subgraph of $H_i$ on $h-h_i$ vertices.
	Now, if $h_i=h-1$ then this subgraph is empty, so
	$|E(D_i) \setminus E(D_{i-1})|=|E(H_i)|=m=(h-1)a(H)=h_ia(H)$. Otherwise, by the definition of $a(H)$,
	this subgraph has at most $(h-h_i-1)a(H)$ edges, so
	$|E(D_i) \setminus E(D_{i-1})| \ge m - (h-h_i-1)a(H) = h_ia(H)$.
	We have now shown that
	$$
	|E(D)| = m + \sum_{i=2}^t |E(D_i) \setminus E(D_{i-1})| \ge m +  \sum_{i=2}^t a(H)h_i =
	(\sum_{i=1}^t a(H)h_i)-a(H) = a(H)(|V(D)|-1)\;.
	$$
	The probability that $G$ has a particular subgraph with $|V(D)|$ vertices and at least
	$a(H)(|V(D)|-1)$ edges is at most 
	$$
	n^{|V(D)|}p^{a(H)(|V(D)|-1)} = n^{-|V(D)|(\frac{a(H)}{a^*}-1)+\frac{a(H)}{a^*}} \le
	n^{-k_0(\frac{a(H)}{a^*}-1)+\frac{a(H)}{a^*}} = o_n(1)
	$$
	where the last inequality follows by setting $k_0=\lceil a(H)/(a(H)-a^*) \rceil$.
\end{proof}

\begin{proof}[Completing the proof of Theorem \ref{t:2}]
	Since $\D \cup D^*$ consists of only a bounded number of elements (at most $2^{k_0^2h^2+1}$),
	and since by Lemmas \ref{l:none-1} and \ref{l:none-2} each element of $\D \cup D^*$ is a subgraph of $G$ with probability $o_n(1)$, we have that almost surely $G$ has no element of $\D \cup D^*$ as a subgraph.
	In particular, almost surely $G_H$ has no element of $\D \cup D^*$ as a subgraph.
	By Lemma \ref{l:config}, almost surely $G_H$ is a dag.
	
\end{proof}

\section{Skewness}\label{s:skew}

For a vertex coloring of a graph, we say that a permutation of the vertices {\em respects} the coloring, if
all the vertices of any given color are consecutive. As mentioned in the introduction, the skewness of $H$, denoted by $s(H)$ is a central parameter of Theorem \ref{t:3}. Here is its definition.

\begin{definition}[Skewness]
	Let $H$ be a digraph. For a coloring $C$ of the vertices of $H$, let
	$s_H(C)=\max_{\pi} |E(H_L(\pi))|$ where the maximum is taken over all permutations of the vertices of $H$
	that respect $C$ (in particular, in any such permutation, we are guaranteed that the number of edges going from left to right is at most $s_H(C)$).
	The {\em skewness} of $H$, is
	$$
	s(H)=\min_{C} \, s_H(C)
	$$
	where the minimum is taken over all vertex colorings.
\end{definition}
Notice that $\lceil |E(H)|/2 \rceil \le s(H) \le |E(H)|$ as for any permutation $\pi$ it holds that
$|E(H_L(\pi))| \ge |E(H)|/2$ or else $|E(H_L(\pi^{rev}))| \ge |E(H)|/2$ (note: $\pi$ respects $C$ if and only if $\pi^{rev}$ respects $C$).
Although the definition of skewness applies to every digraph, we are interested in the case where $H$ is a dag.

\noindent {\bf Examples:}
\begin{itemize}
	\item {\em Rooted stars}.
	If $H$ is a star rooted at some vertex $v$ then consider any vertex coloring
	$C$. Suppose, wlog, that $v$ is a source. We can always find a permutation $\pi$ respecting $C$ in which $v$ is the first vertex. In such a permutation, $E(H_L(\pi))=E(H)$, so $s_H(C)=|E(H)|$.
	Thus, $s(H)=|E(H)|$.
	\item {\em Dags that are not rooted stars}.
	If $H$ is a dag that is not a rooted star then one of the following holds: Either $H$ has two disjoint edges, or else $H$ has a directed path on two edges.
	Consider first the case that $H$ has two disjoint edges, say $(a,b)$ and $(c,d)$. Color $\{a,d\}$ red
	and color $\{b,c\}$ blue (any remaining vertices may be colored arbitrarily). Then, in any
	permutation $\pi$ respecting such a coloring, $|E(H_R(\pi))| \ge 1$. Hence $s(H) \le |E(H)|-1$.
	Consider next the case that $H$ has a directed path on two edges, say $(a,b)$ and $(b,c)$ are the edges of
	such a path. Color $\{a,c\}$ red and color $b$ blue (any remaining vertices may be colored arbitrarily).
	Then, in any permutation $\pi$ respecting such a coloring, $|E(H_R(\pi))| \ge 1$. Hence $s(H) \le |E(H)|-1$.
	We have proved that if $H$ is a dag that is not a rooted star, then $s(H) \le |E(H)|-1$.
	\item {\em Balanced bipartite dags}.
	A balanced bipartite dag is a bipartite dag with a bipartition in which
	precisely half of the edges go from one part to the other (for example, a directed path with an even number of edges satisfies this requirement). Let $H$ be such a dag. Consider a coloring in which the vertices of one
	part are red and the vertices of the other part are blue.  Then, in any
	permutation $\pi$ respecting such a coloring, $|E(H_L(\pi))| = |E(H)|/2$. Hence $s(H)=|E(H)|/2$.
	Notice that if $H$ is a bipartite dag with an odd number of edges where the number of edges going from some part to the other is $\lfloor |E(H)|/2 \rfloor$, then a similar argument shows that
	$s(H)=\lceil |E(H)|/2 \rceil$.
	\item {\em Transitive tournaments}. Suppose that $H=T_h$ where the vertices are labeled $\{1,\ldots,h\}$
	and the edges are $(i,j)$ for $1 \le i < j \le h$. Assume first that $h$ is even.
	Consider the coloring with $h/2$ colors whose color classes are $\{i,h+1-i\}$ for $i=1,\ldots,h/2$.
	In any permutation $\pi$ respecting this coloring, precisely half of the edges connecting vertices in distinct vertex classes go from right to left, so we have that $|E(H_L(\pi))| \le h/2+ (\binom{h}{2}-h/2)/2=h^2/4$.
	It therefore holds that $s(T_h) \le h^2/4$. Assume next that $h$ is odd. Consider the coloring with
	$(h+1)/2$ colors whose color classes are $\{i,h+1-i\}$ for $i=1,\ldots,(h+1)/2$ (observe that the last color class only consists of vertex $(h+1)/2$). In any permutation $\pi$ respecting this coloring,
	precisely half of the edges connecting vertices in distinct vertex classes go from right to left, so we have
	that $|E(H_L(\pi))| \le (h^2-1)/4$. It therefore holds that $s(T_h) \le (h^2-1)/4$. Observe, in particular,
	that $s(T_3)=2$ and, more generally, $s(T_h)=(1+o_h(1))|E(T_h)|/2$.
\end{itemize}

We have already observed that $\lceil |E(H)|/2 \rceil \le s(H) \le |E(H)|-1$ if $H$ is not a rooted star.
However, in order to prove that Theorem \ref{t:3} is a significant sharpening of Theorem, \ref{t:1} we would like to prove  that $s(H)$ is just barely above $|E(H)|/2$ for all dags which are not too sparse. This is proved in the following lemma.
\begin{lemma}\label{l:skew}
	Let $H$ be a dag with $h$ vertices and $m$ edges. Then $s(H)=m/2+O(m^{4/5}h^{1/5})$.
\end{lemma}
\begin{proof}
	Let $k$ be a positive integer parameter to be set later.
	We consider a random coloring of $V(H)$ with the set of colors $[k]$, where each vertex uniformly and independently chooses its color. Let the color color classes be $A_1,\ldots,A_k$ where $A_i$ is the set of vertices colored with color $i$.
	
	An edge $e \in E(H)$ is an {\em inside edge} if both of its endpoints are colored the same.
	Let $F$ be the set of inside edges. We have that ${\mathbb E}[|F|]=m/k$.
	By Markov's inequality
	\begin{equation}\label{e:markov}
	\Pr[|F| \ge 2m/k] \le \frac{1}{2}\;.
	\end{equation}
	
	Let $F(i,j)$ be the set of edges going from $A_i$ to $A_j$.
	Clearly, ${\mathbb E}[|F(i,j)|]=m/k^2$. We also need to upper bound the probability that $|F(i,j)|$ deviates significantly from its expected value. As we will use Chebyshev's inequality for this task, we present
	$|F(i,j)|$ as the sum of $m$ indicator random variables $X(a,b)$ for each $(a,b) \in E(H)$, where
	$X(a,b)=1$ if and only if $a \in A_i$ and $b \in A_j$.
	To upper-bound $Var[|F(i,j)|]$ we must therefore upper-bound $Cov(X(a,b),X(c,d))$ for a pair of distinct edges $(a,b),(c,d) \in E(H)$. Now, if $\{a,b\} \cap \{c,d\} = \emptyset$ then $X(a,b),X(c,d)$ are independent.
	Otherwise, $|\{a,b\} \cup \{c,d\}|=3$ and we have that
	$Cov(X(a,b),X(c,d)) \le  \Pr[(X(a,b) = 1) \cap (X(c,d) =1)] \le 1/k^3$. As a graph with $h$ vertices and
	$m$ edges has fewer than $2hm$ ordered pairs of edges with a common endpoint, we have that
	$$
	Var[|F(i,j)|] \le {\mathbb E}[|F(i,j)|]+ \frac{2hm}{k^3} = \frac{m}{k^2}+ \frac{2hm}{k^3} \le \frac{3hm}{k^3}\;.
	$$
	By Chebyshev's inequality, we have that
	$$
	\Pr\left[|F(i,j) - \frac{m}{k^2}| \ge \sqrt{\frac{6hm}{k}}\right] \le \frac{3hm}{k^3} \cdot \frac{1}{(6hm)/k} = \frac{1}{2k^2}\;.
	$$
	As there are at most $k(k-1)$ ordered pairs $(i,j)$ with $i \neq j$ and $1 \le i, j \le k$, we have by the last inequality, by \eqref{e:markov}, and the union bound that with probability at least $1-(1/2)-k(k-1)/(2k^2) > 0$, it holds that $|F| \le 2m/k$ and for all $i,j$ with $i \neq j$, $|F(i,j) - m/k^2| \le \sqrt{6hm/k}$.
	So, hereafter we assume that this is the case for our coloring $C$.
	
	Consider any permutation $\pi$ respecting $C$. Then
	\begin{align*}
	|E(H_L(\pi))| & \le |F| + \sum_{1 \le i < j \le k} \max\{|F(i,j)|\,,\,|F(j,i)|\}\\
	& \le \frac{2m}{k}+ \frac{k(k-1)}{2}\left(\frac{m}{k^2}+\sqrt{\frac{6hm}{k}}\right)\\
	& \le \frac{m}{2}+\frac{2m}{k}+2k^{3/2}h^{1/2}m^{1/2}\;.
	\end{align*}
	Choosing $k = \lceil (m/h)^{1/5} \rceil$ we obtain from the last inequality that
	$|E(H_L(\pi))| \le m/2 + O(m^{4/5}h^{1/5})$, proving that $s_H(C) \le m/2+O(m^{4/5}h^{1/5})$,
	whence $s(H) = m/2+O(m^{4/5}h^{1/5})$.
\end{proof}

\section{Proof of Theorem \ref{t:3}}\label{s:t-3}

Throughout this section we fix a dag $H$ that is not a rooted star and fix $a^* > a(H)$.
Let $h=|V(H)|$ and $m=|E(H)|$. By the definition of $a(H)$, any subgraph of $H$ with $2 \le t \le h$
vertices has at most $(t-1)a(H)$ edges and in particular, $m \le a(H)(h-1)$.

Let $c^* = c^*(a^*,H)$ be a small positive constant to be determined later.
Whenever necessary, we assume that $n$ is sufficiently large as a function of $H$ and $a^*$.
Let $[n]$ be the set of the vertices of $G \sim {\vec G}(n,n^{-1/a^*})$.
We must prove that the following holds almost surely: For every set $X$ of $x= \lfloor c^*\log n \rfloor$ permutations of $[n]$, there is an $H$-copy of $G$ such that for each $\pi \in X$, $G_L(\pi)$ contains at most $s(H)$ edges of that copy, where $s(H)$ is the skewness of $H$.

Let $X=\{\pi_1,\ldots,\pi_x\}$.
For nonempty disjoint sets $A,B \subset [n]$, we say that $A \Rightarrow B$
in $\pi_i$, if all elements of $A$ are before all elements of $B$ in $\pi_i$.
For $r$ nonempty disjoint sets of vertices $A_1,\ldots,A_r$ with $A_j \subset [n]$, we say that $A_1,\ldots,A_r$ are {\em consistent} with $X$ if for all $1 \le i \le x$
and for all $1 \le j,j' \le r$ with $j \neq j'$, either $A_j \Rightarrow A_{j'}$ in
$\pi_i$ or $A_{j'} \Rightarrow A_j$ in $\pi_i$.
Stated otherwise, in each permutation of $X$ restricted to $\cup_{\ell=1}^r A_\ell$, all the elements of each $A_j$ are consecutive.

\begin{lemma}\label{l:consistent}
	Let $r=2^t$ for some positive integer $t$ and let $n$ be a multiple of $r^x$ where $x$ is a positive integer.
	Suppose that $X=\{\pi_1,\ldots,\pi_x\}$ is a set of permutations of $[n]$.
	There exist disjoint sets $A_1,\ldots,A_r$ with $A_j \subset [n]$
	that are consistent with $X$. Furthermore, $|A_j| = n/r^x$ for $1 \le j \le r$.
\end{lemma}
\begin{proof}
	We first prove the lemma for $t=1$. Namely, we prove that there are two sets $A_1,A_2$
	that are consistent with $X$, each of them containing $n/2^x$ vertices.
	We construct the claimed sets $A_1,A_2$ inductively.
	Let $X_i=\{\pi_1,\ldots,\pi_i\}$. We construct disjoint sets $A_{1,i},A_{2,i}$
	that are consistent with $X_i$ and each of them has size at least $n/2^i$.
	As $X=X_x$, the lemma will follow for $t=1$.
	
	Starting with $X_1=\{\pi_1\}$, we set $A_{1,1}$ to be the first $n/2$ elements of $\pi_1$ and
	set $A_{2,1}$ to be the last $n/2$ elements of $\pi_1$. Hence $A_{1,1} \Rightarrow A_{2,1}$ and both are
	of the required size. Assume that we have already defined $A_{1,i}$, $A_{2,i}$ that are consistent with
	$X_i$ and each have size $n/2^i$. Let $\sigma$ denote the restriction of $\pi_{i+1}$ to $A_{1,i} \cup A_{2,i}$,
	so $\sigma$ is a permutation of $A_{1,i} \cup A_{2,i}$ of order $n/2^{i-1}$.
	Let $C$ denote the first $n/2^i$ elements of $\sigma$ and let $D$ be the remaining $n/2^i$ elements of $\sigma$. Now, suppose first that $|A_{1,i} \cap C| \ge n/2^{i+1}$. Then we must have that
	$|A_{2,i} \cap D| \ge n/2^{i+1}$ so we may set $A_{1,i+1}$ to be any subset of $A_{1,i} \cap C$ of size
	$n/2^{i+1}$ and $A_{2,i+1}$ to be any subset of $A_{2,i} \cap D$ of size $n/2^{i+1}$.
	Otherwise, we must have $|A_{1,i} \cap D| \ge n/2^{i+1}$. Then we must have that $|A_{2,i} \cap C| \ge n/2^{i+1}$ so we may set $A_{1,i+1}$ to be any subset of $A_{1,i} \cap D$ of size $n/2^{i+1}$ and
	set $A_{2,i+1}$ to be any subset of $A_{2,i} \cap C$ of size $n/2^{i+1}$. Observe that in any case, we have that either $A_{1,i+1} \Rightarrow A_{2,i+1}$ in $\pi_{i+1}$ or else $A_{2,i+1} \Rightarrow A_{1,i+1}$
	in $\pi_{i+1}$. Hence, $A_{1,i+1},A_{2,i+1}$ are consistent with $X_{i+1}$ and are of the required size.
	
	Having proved the base case $t=1$ of the lemma, we prove the general case by induction.
	So, assume that $r=2^t$ with $t > 1$ and that we already found disjoint subsets
	$A_1,\ldots,A_{r/2}$, each of size $n/(r/2)^x$ that are consistent with $X$.
	Applying the case $t=1$ to each $A_j$ separately as the ground set, we can find, for each $1 \le j \le r/2$,
	two disjoint sets of $A_j$, say $B_j$ and $C_j$, each of size $|A_j|/2^x=n/r^x$ such that
	$B_j,C_j$ are consistent with $X |_{A_j}$.
	Hence, $B_1,C_1,B_2,C_2,\ldots,B_{r/2},C_{r/2}$ are $r$ disjoint subsets of $[n]$ that are consistent with $X$
	and each is of the required size $n/r^x$.
\end{proof}

Our approach to proving Theorem \ref{t:3} is to show that almost surely, there is a labeled $H$-copy in $G$
that is suitably embedded inside $r$ disjoint subsets of vertices that are consistent with $X$.
As we have no control over the chosen $X$ (but we do know that it contains $x= \lfloor c^*\log n \rfloor$ permutations), we also do not have control over the $r$ disjoint subsets consistent with it (but do have control on their size, by Lemma \ref{l:consistent}), so we must guarantee that there are many labeled $H$-copies
such that, almost surely, no choice of $r$ disjoint subsets can avoid a labeled $H$-copy.
We next formalize our arguments.

By the definition of $s(H)$, there is a vertex coloring $C$ of $V(H)$ such that for any permutation $\pi$
of $V(H)$ which respects $C$, it holds that $|E(H_L(\pi))| \le s(H)$. Now, suppose that $C$ uses $r$ colors.
We may assume that $r=2^t$ for some positive integer $t$, as otherwise we can just add some dummy unused colors.
As this assumption can only increase the number of colors by less than a factor of $2$, we have that
$r \le 2h-2$. For $1 \le i \le r$, let $W_i \subset V(H)$ be the vertices colored with color $i$
(possibly $W_i = \emptyset)$ and observe that $\cup_{i=1}^r W_i = V(H)$. Also note that in any
permutation of $V(H)$ that respects $C$, the vertices of $W_i$ are consecutive.

Recall that a labeled $H$-copy $H^*$ of $G$ is synonymous with an injective mapping $\phi: V(H) \rightarrow V(G)$
such that $(u,v) \in E(H)$ implies that $(\phi(u),\phi(v)) \in E(G)$.
Let $(V_1,\ldots,V_r)$ be an $r$-tuple of disjoint sets of vertices of $G$.
We say that an $H$-copy $H^*$ of $G$ is {\em consistent} with $(V_1,\ldots,V_r)$ if
$\phi(W_i) \subseteq V_i$ for $1 \le i \le r$. The following lemma shows that, almost surely,
for all $r$-tuples $(V_1,\ldots,V_r)$ in which all $V_i$'s are large, there is an $H$-copy in $G$ consistent with the $r$-tuple.

\begin{lemma}\label{l:main-t3}
	There exists $\alpha=\alpha(a^*,H) < 1$ such that almost surely, for all $r$-tuples $(V_1,\ldots,V_r)$ of disjoint sets of vertices of $G$ where $|V_i| = \lfloor n^\alpha \rfloor$, there is an $H$-copy of $G$ consistent with $(V_1,\ldots,V_r)$.
\end{lemma}
\begin{proof}
	Since $a^* > a(H)$ we may fix a constant $\alpha$ such that
	\begin{equation}\label{e:alpha}
		1 > \alpha > \frac{a(H)}{a^*}\;.
	\end{equation}
	First observe that the number of
	possible $r$-tuples satisfying the lemma's assumption is at most
	$$
	(n^{n^\alpha})^r < n^{2h n^\alpha}\;.
	$$
	where we have used that $r < 2h$.
	So, fixing such an $r$-tuple $(V_1,\ldots,V_r)$, it suffices to prove that the probability that
	there is no labeled $H$-copy consistent with it is $o(n^{-2h n^\alpha})$.
	
	Recall that $W_i \subset V(H)$ is the set of vertices of $H$ colored with color $i$.
	Now, consider an $r$-tuple $U=(U_1,\ldots,U_r)$ with $U_i \subset V_i$ and $|U_i|=|W_i|$.
	We will also view $U$ as the set $\cup_{i=1}^r U_i$ so $|U|=h$.
	Let $A(U)$ be the event that there is an $H$-copy $\phi$ in $G$ such that
	$\phi(W_i)=U_i$ for $1 \le i \le r$ (notice that if this event holds, then there is an $H$-copy
	consistent with $(V_1,\ldots,V_r)$). We have that
	\begin{equation}\label{e:prau}
		p^m \le \Pr[A(U)] \le h! p^m
	\end{equation}
	as there are at most $\pi_{i=1}^r {|U_i|!} \le h!$ possible mappings to consider and for each of them, the probability of being an $H$-copy is precisely $p^m$. So, our goal is to prove that
	$$
		\Pr[\cap_{U} \overline{A(U)}] = o(n^{-2h n^\alpha})
	$$
	where the intersection is over all possible choices of $U$.
	
	For two events $A(U)$ and $A(U')$ with $U \neq U'$, we write $A(U) \sim A(U')$
	if $|U \cap U'| \ge 2$. Observe that if $A(U) \nsim A(U')$
	then $A(U)$ and $A(U')$ are independent, as they involve disjoint sets of possible edges.
	Let
	$$
		\Delta = \sum_{U \sim U'} \Pr[A(U) \cap A(U')],~~~\mu = \sum_{U} \Pr[A(U)]
	$$
	where the sum for $\Delta$ is over ordered pairs.
	We will upper-bound $\Delta$ and lower-bound $\mu$ so that we will be able to fruitfully apply Janson's inequality.
	First, observe that the number of possible choices for $U$ is precisely
	$$
		\prod_{i=1}^r \binom{|V_i|}{|U_i|} = \prod_{i=1}^r \binom{\lfloor n^\alpha \rfloor}{|W_i|} \ge \binom{\lfloor n^\alpha \rfloor}{h} 
	$$
	where the last inequality follows from the fact that $\sum_{i=1}^r|W_i|=h$.
	It follows from (\ref{e:prau}) and the last inequality that
	\begin{equation}\label{e:mu}
		\mu \ge \binom{\lfloor n^\alpha \rfloor}{h} p^m \ge h^{-h} n^{\alpha h-m/a^*}
		\ge h^{-h} n^{\alpha h-a(H)(h-1)/a^*} \ge h^{-h} n^{h(\alpha-\frac{a(H)}{a^*})+\frac{a(H)}{a^*}} \;.
	\end{equation}
	To evaluate $\Delta$, observe that if $U \sim U'$, then $2 \le |U \cap U'| \le h-1$, so we may partition the
	terms in the definition of $\Delta$ to $h-2$ parts according to the size of the $U \cap U'$.
	Let $2 \le t \le h-1$. We first estimate the number of ordered pairs $U,U'$ with $|U \cap U'|=t$.
	The number of choices for $U$ is less than $|\cup_{i=1}^r{V_i}|^h \le (rn^\alpha)^h$.
	As $U'$ contains $h-t$ other vertices not in $U$, but in $\cup_{i=1}^r{V_i}$, there are less than
	$|\cup_{i=1}^r{V_i}|^{h-t}$ choices for these vertices. As for the $t$ remaining vertices of the intersection
	$U \cap U'$, they are all taken from the $h$ vertices of $U$, so there are fewer than $h^t$ options for them.
	It follows that the number of ordered pairs $U,U'$ with $|U \cap U'|=t$ is less than
	$$
		(rn^\alpha)^h (rn^\alpha)^{h-t} h^t =h^t (rn^\alpha)^{2h-t}\;.
	$$
	Now, consider an $H$-copy $H_1$ on the vertices of $U$ and an $H$-copy $H_2$ on the vertices of $U'$.
	The number of edges of $H_1$ is $m$. The number of edges of $H_2$ with both endpoints in $U \cap U'$
	is at most $(t-1)a(H)$, by the definition of $a(H)$. Hence the number of edges of $H_1 \cup H_2$
	is at least $2m-(t-1)a(H)$. The probability that $G$ contains a labeled subgraph on this amount of edges is
	therefore at most $p^{2m-(t-1)a(H)}$. As there are at most $h!$ choices for $H_1$ and at most $h!$ choices for
	$H_2$, we have that
	$$
		\Pr[A(U) \cap A(U')] \le (h!)^2 p^{2m-(t-1)a(H)}\;.
	$$
	It follow that
	$$
		\Delta \le \sum_{t=2}^{h-1} h^t (rn^\alpha)^{2h-t} (h!)^2 p^{2m-(t-1)a(H)}=
		\sum_{t=2}^{h-1} (h!)^2 h^t r^{2h-t} n^{2h\alpha-\frac{2m}{a^*}-\frac{a(H)}{a^*}+t(\frac{a(H)}{a^*}-\alpha)}\;.
	$$
	By \eqref{e:alpha}, the largest summand occurs when $t$ is smallest, i.e. when $t=2$. Thus, we get that
	\begin{equation}\label{e:delta}
		\Delta \le (h-2)(h!)^2 h^t r^{2h-t} n^{2(h-1)\alpha-\frac{2m}{a^*}+\frac{a(H)}{a^*}}
		< (2h)^{4h+1} n^{2(h-1)\alpha-\frac{2m}{a^*}+\frac{a(H)}{a^*}}\;.
	\end{equation}
We first consider the case where $\Delta \le \mu$.
Observe that by \eqref{e:mu} we have that
$$
\mu \ge  h^{-h} n^{h(\alpha-\frac{a(H)}{a^*})+\frac{a(H)}{a^*}}
$$
so we have by Janson's inequality (see \cite{AS-2004} Theorem 8.1.1) and the last inequality that
$$
	\Pr[\cap_{U} \overline{A(U)}] \le e^{-\mu/2}  \le  e^{-0.5h^{-h} n^{h(\alpha-\frac{a(H)}{a^*})+\frac{a(H)}{a^*}}} = o(n^{-2h n^\alpha})
$$
where in the last inequality we have used that
$$
h(\alpha-\frac{a(H)}{a^*})+\frac{a(H)}{a^*} > \alpha
$$
which indeed holds by \eqref{e:alpha}.

Consider next the case where $\Delta \ge \mu$.
By \eqref{e:mu} and \eqref{e:delta} we have that
$$
\frac{\mu^2}{2\Delta} \ge \frac{h^{-2h}n^{2\alpha h - \frac{2m}{a^*}}}{2(2h)^{4h+1}n^{2(h-1)\alpha-\frac{2m}{a^*}+
	\frac{a(H)}{a^*}}}=\Theta(n^{2\alpha-\frac{a(H)}{a^*}})
$$
so we have by Janson's inequality (see \cite{AS-2004} Theorem 8.1.2) and the last inequality that
$$
\Pr[\cap_{U} \overline{A(U)}] \le e^{-\frac{\mu^2}{2\Delta}} =  o(n^{-2h n^\alpha})
$$
where in the last inequality we have used that
$$
2\alpha-\frac{a(H)}{a^*} > \alpha
$$
which indeed holds by \eqref{e:alpha}.
\end{proof}

\begin{proof}[Completing the proof of Theorem \ref{t:3}]
	Let $\alpha$ be the constant from Lemma \ref{l:main-t3}. Define
	$$
	c^* = \frac{(1-\alpha)}{2\log r}
	$$
	and recall that $x= \lfloor c^*\log n \rfloor$.
	By Lemma \ref{l:main-t3}, $G \sim \vec{G}(n,p)$ almost surely satisfies the property that for all
	$r$-tuples $(V_1,\ldots,V_r)$ of disjoint sets of vertices of $G$ where $|V_i| = \lfloor n^\alpha \rfloor$, there is an $H$-copy of $G$ consistent with $(V_1,\ldots,V_r)$. So, hereafter we assume that $G$ indeed satisfies this property. We must show that for every set $X=\{\pi_1,\ldots,\pi_x\}$ of
	permutations of $[n]$, there is an $H$-copy of $G$ such that each element $\pi \in X$ it holds that
	$G_L(\pi)$ contains at most $s(H)$ edges of that copy. So, hereafter we fix an arbitrary $X=\{\pi_1,\ldots,\pi_x\}$ and show that such an $H$-copy exists.
	
	Let $n-r^x < N \le n$ where $N$ is a multiple of
	$r^x$. Recall also that $r=2^t$ for some positive integer $t$ and that $r \le 2h-2$.
	By Lemma \ref{l:consistent} there exist disjoint sets $A_1,\ldots,A_r$ with $A_j \subset [N] \subseteq [n]$
	that are consistent with $X$ (in fact, already consistent with the restriction of each $\pi_i$ to $[N]$) and furthermore, $|A_j|=N/r^x$.
	We claim that, in fact, $|A_j| \ge \lfloor n^\alpha \rfloor$.
	Observe indeed that by the definition of $c^*$, we have that
	$$
	x= \lfloor c^*\log n \rfloor \le \frac{(1-\alpha)\log n -1} {\log r}
	$$ 
	which implies  that
	$$
	r^x \le \frac{n^{1-\alpha}}{2} \le \frac{n}{n^\alpha +1}
	$$
	implying that
	$$
	\frac{N}{r^x} \ge \frac{n-r^x}{r^x} \ge \frac{n}{r^x}-1 \ge n^\alpha\;.
	$$
	In particular, we can fix subsets $V_i \subseteq A_i$ with $|V_i|=\lfloor n^\alpha \rfloor$.
	By the property of $G$, we have that there is an $H$-copy of $G$ that is consistent with
	$(V_1,\ldots,V_r)$. But this means that for each $\pi_i \in X$, when restricted to the vertices of that $H$-copy, all vertices of a given color class of the coloring $C$ are consecutive.
	By the definition of skewness, $G_L(\pi_i)$ contains at most $s(H)$ edges of that copy.
\end{proof}

\section{Concluding remarks and open problems}\label{s:final}

We have proved that for every totally balanced dag $H$ that is not a rooted star, the exponent $-1/a(H)$ is a threshold for covering the $H$-copies of a digraph by directed acyclic subgraphs. Namely, for every $a^* > a(H)$ it holds that
almost surely that $\tau(H,G)=\Theta(\log n)$ while for every $a^* < a(H)$ it holds almost surely that
$\tau(H,G) \le 1$, where $G \sim \vec{G}(n,n^{-1/a^*})$. By Proposition \ref{prop:1}, this determines the
correct threshold exponent for almost all \footnote{By ``almost all'' mean that a random graph $G(h,\frac{1}{2})$ is almost surely totally balanced (hence all of its acyclic orientations are totally balanced dags).} dags.

The first natural question is whether the same result holds for all dags that are not rooted stars. Let us first note that even though
$\tau(H,G)$ is a monotone parameter, it is not obvious that
a threshold exponent even exists in the above sense (namely, jumping from $\tau(H,G) = \Theta(\log n)$ to $\tau(H,G) \le 1$). We next show that the answer is generally false: there are some dags for which $-1/a(H)$ is not the threshold exponent. Let $H$ be the dag from Figure \ref{f:H}.
Clearly $a(H)=3/2$. However:
\begin{figure}
	\includegraphics[scale=0.8,trim=-80 420 500 20, clip]{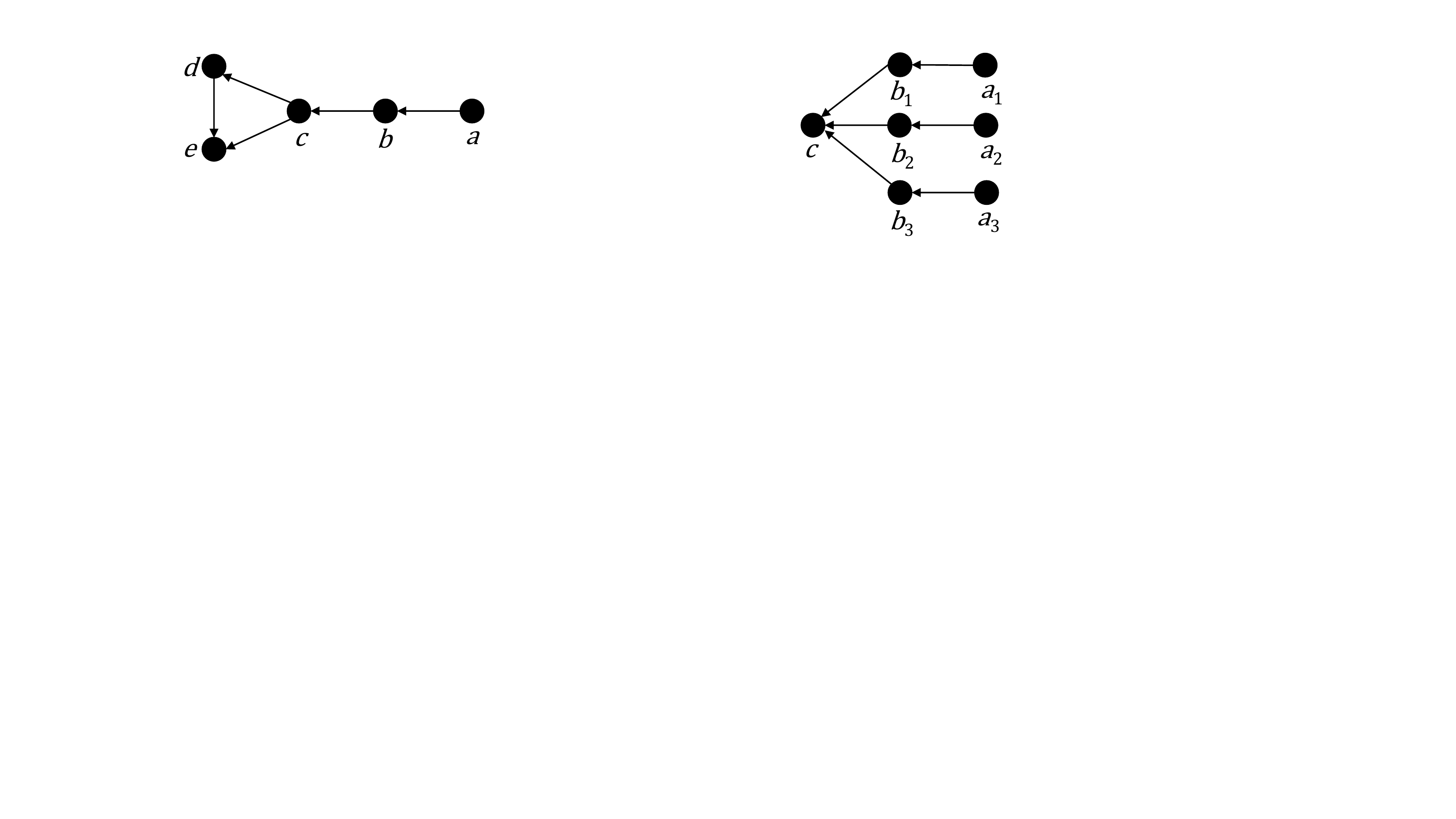}
	\caption{A dag $H$ for which $-1/a(H)$ is not the threshold exponent.}
	\label{f:H}
\end{figure}

\begin{proposition}\label{prop:h}
	Let $H$ be the dag from Figure \ref{f:H}. Then for all $a^* > 4/3$ it holds for $G \sim \vec{G}(n,n^{-1/a^*})$
	that $\tau(H,G)=\Theta(\log n)$.
\end{proposition}
\begin{proof}[Proof (sketch)]
	By monotonicity, it suffices to prove the proposition for all, say, $3/2 > a^* > 4/3$.
	Fix $3/2 > a^* > 4/3$ and consider $G \sim \vec{G}(n,n^{-1/a^*})$.
	We define four properties, P1, P2, P3, P4 and show that each holds almost surely.
	
	\noindent
	{\em Property P1:} ``The number of directed cycles of length $2$ in $G$ is $\Theta(n^{2-2/a^*})$''. For any
	pair of vertices $u,v$, the probability that both $(u,v)$ and $(v,u)$ are edges is $n^{-2/a^*}$ and there are
	$\binom{n}{2}$ pairs to consider. By Chebyshev's inequality, the probability of deviating from the expected amount $\Theta(n^{2-2/a^*})$ by at most a constant factor is $1-o_n(1)$. Thus, P1 holds almost surely.

	\noindent
	{\em Property P2:} ``The number of subgraphs of $G$ on four vertices and five edges is
	$\Theta(n^{4-5/a^*})$''. Let $K$ be a digraph on four vertices and five edges (note: there are only a constant number of possible digraphs on four vertices and five edges). The number of $K$-copies in $G$
	is $\Theta(n^{4-5/a^*})$. Indeed, for any four labeled vertices, the probability that they contain
	a labeled copy of $K$ is $n^{-5/a^*}$ and there are $\Theta(n^4)$ choices for such four labeled copies.
	Again, by Chebyshev's inequality it is easy to show that the probability of deviating from the expected amount $\Theta(n^{4-5/a^*})$ by at most a constant factor is $1-o_n(1)$. Thus, P2 holds almost surely.
	
	\noindent
	{\em Property P3:} ``The number of vertices that appear in some $T_3$-copy of $G$ as the
	source vertex (e.g. vertices like $c$ in Figure \ref{f:H} inside the $T_3$ induced by $\{c,d,e\}$) is
	$\Theta(n^{3-3/a^*})$.'' This is analogous to the arguments in P1 and P2. Thus, P3 holds almost surely.
	
	\noindent
	{\em Property P4:}
	Let $P_3$ denote the directed path on two edges and three vertices. Fix some $3-3/a^* > \beta > 1/a^*$.
	Observe that $\beta$ exists since $a^* > 4/3$.
	As the fractional arboricity of every tree, in particular, of $P_3$, is $1$, it is not difficult to prove (similar to the proof of Theorem \ref{t:3}) that almost surely, the following property P4 holds:
	``For every subset $U \subset V(G)$  with $|U| \ge n^\beta$ it holds that $\tau(P_3,G[U]) \ge c^*\log n$
	where $c^*$ is a constant depending only on $\beta,a^*$ and where $G[U]$ is the subgraph of $G$ induced by $U$.
	
	Let $G$ be a graph for which P1, P2, P3, P4 hold. Let $U^*$ be the set of vertices of $G$ appearing in some $T_3$-copy of $G$
	as the source vertex. By P3, we have that $|U^*| = \Theta(n^{3-3/a^*})$.
	Remove from $U^*$ all vertices that are contained in directed cycles of length $2$ and also all vertices
	contained is subgraphs on four vertices and five edges, remaining with a subset $U$. By P1 and P2
	we have that $|U| = \Theta(n^{3-3/a^*}) -\Theta(n^{2-2/a^*})-\Theta(n^{4-5/a^*}) = \Theta(n^{3-3/a^*}) \ge n^\beta$. Consider any set $X$ of permutations of $V(G)$ with $|X| < c^*\log n$. 
	Now, by P4, we have
	that there is some $P_3$-copy of $G[U]$ that is not covered by any element of $X$ (viewing each permutation of $X$ as restricted to $U$). Let the vertices of such a $P_3$-copy be $a,b,c$ where $(a,b)$ and $(b,c)$ are edges. Since each vertex of $U$ is a source vertex of some $T_3$-copy, there is a $T_3$-copy of
	$G$ containing $c$. So, let the vertices of this copy be $\{c,d,e\}$ as in Figure \ref{f:H}.
	Now, we must have $b \neq d$ and $b \neq e$ since $c$ is not on any directed cycle of length $2$.
	It must also be that $a \neq d$ and $a \neq e$ since $c$ is not on any subgraph on four vertices and five edges. In any case, we have that $G$ contains an $H$-copy that is uncovered by $X$.
\end{proof}

\begin{problem}\label{prob:1}
	Is there a threshold exponent for all dags that are not rooted stars, and if so, what is it?
\end{problem}
\noindent By Theorem \ref{t:1}, if the answer to Problem \ref{prob:1} is yes, then the threshold exponent is at most $-1/a(H)$.

Another question is whether $s(H)$ (skewness) is the optimal parameter of choice in the statement of Theorem
\ref{t:3}. Perhaps for some graphs even a value smaller than $s(H)$ (but of course at least $m/2$) still satisfies
the statement of the theorem? Clearly, the answer is no whenever $s(H)=\lceil m/2 \rceil$ and also in some other
cases, e.g. when $H$ is obtained from a rooted star with $m \ge 4$ edges by flipping one edge.
The smallest (in terms of $h$) case for which we can ask this question is the transitive tournament $T_4$.
Notice that $s(T_4)=4$ while $m/2=3$ and $a(T_4)=2$.
\begin{problem}
	Does the following hold for all $a^* > 2$:
	There is a constant $c^* = c^*(a^*) > 0$ such that almost surely $G \sim {\vec G}(n,n^{-1/a^*})$ has the property that for every set $X$ of at most $c^*\log n$ permutations, there is a $T_4$-copy of $G$ such that each element of $X$ contains at most three edges of that copy.
\end{problem}


\begin{thebibliography}{10}
	
	\bibitem{AS-2004}
	N.~Alon and J.~Spencer.
	\newblock {\em The probabilistic method}.
	\newblock John Wiley \& Sons, 2004.
	
	\bibitem{AY-1993}
	N.~Alon and R.~Yuster.
	\newblock Threshold functions for {$H$}-factors.
	\newblock {\em Combinatorics, Probability and Computing}, 2(2):137--144, 1993.
	
	\bibitem{CCHZ-2013}
	Y.~Chee, C.~Colbourn, D.~Horsley, and J.~Zhou.
	\newblock Sequence covering arrays.
	\newblock {\em SIAM Journal on Discrete Mathematics}, 27(4):1844--1861, 2013.
	
	\bibitem{ER-1960}
	P.~Erd{\H{o}}s and A.~R{\'e}nyi.
	\newblock On the evolution of random graphs.
	\newblock {\em Publ. Math. Inst. Hung. Acad. Sci}, 5(1):17--60, 1960.
	
	\bibitem{furedi-1996}
	Z.~F{\"u}redi.
	\newblock Scrambling permutations and entropy of hypergraphs.
	\newblock {\em Random Structures \& Algorithms}, 8(2):97--104, 1996.
	
	\bibitem{FHRT-1992}
	Z.~F{\"u}redi, P.~Hajnal, V.~R{\"o}dl, and W.~Trotter.
	\newblock Interval orders and shift graphs.
	\newblock In A.~Hajnal and V.~T. S\'os, editors, {\em Sets, Graphs, and
		Numbers}, volume~60, pages 297--313. North-Holland Publishing Co.; J{\'a}nos
	Bolyai Mathematical Society, 1992.
	
	\bibitem{ishigami-1995}
	Y.~Ishigami.
	\newblock Containment problems in high-dimensional spaces.
	\newblock {\em Graphs and Combinatorics}, 11(4):327--335, 1995.
	
	\bibitem{ishigami-1996}
	Y.~Ishigami.
	\newblock An extremal problem of $d$ permutations containing every permutation
	of every $t$ elements.
	\newblock {\em Discrete Mathematics}, 159(1-3):279--283, 1996.
	
	\bibitem{nash-1964}
	C.~S. J.~A. Nash-Williams.
	\newblock Decomposition of finite graphs into forests.
	\newblock {\em Journal of the London Mathematical Society}, 1(1):12, 1964.
	
	\bibitem{radhakrishnan-2003}
	J.~Radhakrishnan.
	\newblock A note on scrambling permutations.
	\newblock {\em Random Structures \& Algorithms}, 22(4):435--439, 2003.
	
	\bibitem{spencer-1972}
	J.~Spencer.
	\newblock Minimal scrambling sets of simple orders.
	\newblock {\em Acta Mathematica Hungarica}, 22(3-4):349--353, 1972.
	
	\bibitem{tarui-2008}
	J.~Tarui.
	\newblock On the minimum number of completely 3-scrambling permutations.
	\newblock {\em Discrete Mathematics}, 308(8):1350--1354, 2008.
	
	\bibitem{yuster-2020}
	R.~Yuster.
	\newblock Covering small subgraphs of (hyper)tournaments with spanning acyclic
	subgraphs.
	\newblock {\em Electronic Journal of Combinatorics}, 27(4):P4.13, 2020.
	
\end{thebibliography}
\end{document}